\documentclass[12pt,reqno]{amsart} 

\usepackage{amssymb,latexsym}
\usepackage{cite} 

\usepackage[height=190mm,width=130mm]{geometry} 
\usepackage{amsmath}
\usepackage[all]{xy}
\usepackage{amssymb}
\usepackage{amscd}

\date{}
\begin{document}
\title{A Generalization of $J$-Quasipolar Rings}
\author{T. Pekacar Calci}
\address{Tugce Pekacar Calci, Department of Mathematics, Ankara University,
Turkey} \email{tcalci@ankara.edu.tr}

\author{S. Halicioglu}
\address{Sait Halicioglu, Department of Mathematics, Ankara University,
Turkey} \email{halici@ankara.edu.tr}
\author{A. Harmanci}
\address{Abdullah Harmanci, Department of Mathematics, Hacettepe University,
Turkey} \email{harmanci@hacettepe.edu.tr}

\newtheorem{thm}{Theorem}[section]
\newtheorem{lem}[thm]{Lemma}
\newtheorem{prop}[thm]{Proposition}
\newtheorem{cor}[thm]{Corollary}
\newtheorem{df}[thm]{Definition}
\newtheorem{nota}{Notation}
\newtheorem{note}[thm]{Remark}
\newtheorem{ex}[thm]{Example}
\newtheorem{exs}[thm]{Examples}

\begin{abstract} In this paper, we introduce a class of quasipolar rings
 which  is a generalization of $J$-quasipolar rings. Let $R$ be a
ring with identity. An element $a \in R$ is called {\it
$\delta$-quasipolar} if there exists $p^2 = p\in comm^2(a)$ such
that $a + p$ is contained in $\delta(R)$, and the ring $R$ is
called {\it $\delta$-quasipolar} if every element of $R$ is
$\delta$-quasipolar. We use  $\delta$-quasipolar rings to extend
some results of $J$-quasipolar rings. Then some of the main
results of $J$-quasipolar rings are special cases of our results
for this general setting. We give many characterizations and
investigate general properties of $\delta$-quasipolar rings.
\end{abstract}


\maketitle

\section{Introduction}
Throughout this paper all rings are associative with identity
unless otherwise stated. Let $R$ be a ring. According to Koliha
and Patricio \cite{KP}, the \emph{commutant} and \emph{double
commutant} of an element $a\in R$ are defined by $comm(a)=\{x\in
R~|~xa=ax\}$, $comm^2(a)=\{x\in R~|~xy=yx~\mbox{for all}~y\in
comm(a)\}$, respectively. If $R^{qnil}=\{a\in R~|~1+ax\in
U(R)~\mbox{for every}~x\in comm(a)\}$ and $a\in R^{qnil}$,  then
$a$ is said to be \emph{quasinilpotent} (see \cite{Ha}). The
element $a$ is called {\it quasipolar} if there exists $p^2 = p\in
R$ such that $p\in comm^2(a)$, $a + p$ is invertible in $R$ and
$ap\in R^{qnil}$. Any idempotent $p$ satisfying the above
conditions is called a {\it spectral idempotent} of $a$, and this
term is borrowed from spectral theory in Banach algebra and it is
unique for $a$.

Quasipolar rings have been studied by many ring theorists (see
\cite{CC1}, \cite{GHH}, \cite{GHH1},
 \cite{Ha}, \cite{KP} and \cite{YC}). In \cite{GHH1},  the element $a\in R$ is called
\textit{nil-quasipolar} if there
 exists $p^2=p\in comm^2(a)$
such that $a+p$ is nilpotent, the idempotent $p$ is called a
\textit{nil-spectral idempotent} of $a$. The ring $R$ is said to
be \textit{nil-quasipolar} if every element of $R$ is
nil-quasipolar.  Recently, $J$-quasipolar rings are studied in
\cite{CC3}.
 The element $a$ is called {\it
$J$-quasipolar} if there exists $p^2 = p\in R$ such that $p\in$
$comm^2(a)$ and $a + p\in J(R)$,  $p$ is called a
\textit{$J$-spectral idempotent} of $a$. The ring $R$ is said to
be {\it $J$-quasipolar} if every element of $R$ is $J$-quasipolar.
 Motivated by these, we introduce a new class of
quasipolar rings which  is a generalization of $J$-quasipolar
rings. By using $\delta$-quasipolar rings, we extend some results
of $J$-quasipolar rings.

An outline of the paper is as follows: The second section deals
with $\delta$-quasipolar rings. We prove various basic
characterizations and properties of $\delta$-quasipolar rings.
 We
investigate conditions under which a $\delta$-quasipolar ring is
quasipolar and a quasipolar ring is $\delta$-quasipolar (Corollary
\ref{abdelta} and Proposition \ref{locquasi}). It is proven that
every $J$-quasipolar ring is $\delta$-quasipolar. We supply an
example to show that all $\delta$-quasipolar rings need not be
$J$-quasipolar. Among others the $\delta$-quasipolarity of Dorroh
extensions and some classes of matrix rings are investigated. In
Section 3, we introduce an upper class of $\delta$-quasipolar
rings, namely, weakly $\delta$-quasipolar rings. We show that
every direct summand of a weakly $\delta$-quasipolar ring is
weakly $\delta$-quasipolar and every direct product of weakly
$\delta$-quasipolar rings is weakly $\delta$-quasipolar, and we
give some properties of such rings.

In what follows, $\Bbb{Z}$ and $\Bbb Q$ denote the ring of
integers and the ring of rational numbers and for a positive
integer $n$, $\Bbb{Z}_n$ is the ring of integers modulo $n$. For a
positive integer $n$, let $Mat_{n}(R)$ denote the ring of all
$n\times n$ matrices and $T_n(R)$  the ring of all $n\times n$
upper triangular matrices with entries in $R$. We write $J(R)$ and
$nil(R)$ for the Jacobson radical of $R$ and the set of nilpotent
elements of $R$,  respectively.

\section{$\delta$-quasipolar Rings}

In this section we introduce  the concept of $\delta$-quasipolar
rings and investigate some properties of such rings. We show that
every quasipolar ring need not be $\delta$-quasipolar (Example
\ref{aksi}). It is proven that every $J$-quasipolar ring is
$\delta$-quasipolar and the converse does not hold in general (see
Example \ref{ilki}). Among others we extend some results of
$J$-quasipolar rings for this general setting.

A right ideal $I$ of the ring $R$ is said to be {\it
$\delta$-small} in $R$ if whenever $R=I+K$ with $R/K$ singular
right $R$-module for any right ideal $K$ then $R=K$. In \cite{Zh},
the ideal $\delta(R)$ is introduced as a sum of $\delta$-small
right ideals of $R$.  We begin with the equivalent conditions for
$\delta(R)$ which is proved in \cite[Theorem 1.6]{Zh} for an easy
reference for the reader.

\begin{lem}\label{besli} Given a ring $R$, each of the following sets is equal to
$\delta(R)$.\begin{enumerate} \item[(1)] $R_1$ = the intersection
of all essential maximal right ideals of $R$.\item[(2)] $R_2$ =
the unique largest $\delta$-small right ideal of $R$.\item[(3)]
$R_3 = \{x\in R\mid xR + K_R = R$ implies $K_R$ is a direct
summand of $R_R\}$.\item[(4)] $R_4 = \bigcap\{$ideals $P$ of
$R\mid R/P$ has a faithful singular simple module$\}$.\item[(5)]
$R_5 = \{x\in R\mid$ for all $y\in R$ there exists a semisimple
right ideal $Y$ of $R$ such that $(1 + xy)R\oplus Y = R_R\}$.
\end{enumerate}
\end{lem}

Now we give our main definition.

\begin{df}{\rm Let $R$ be a ring. An element $a \in R$ is called {\it
$\delta$-quasipolar} if there exists $p^2 = p\in comm^2(a)$ such
that $a + p\in \delta(R)$ and  $p$ is called a {\it
$\delta$-spectral idempotent}. The ring $R$ is called {\it
$\delta$-quasipolar} if every element of $R$ is
$\delta$-quasipolar.}
\end{df}

The following are examples for $\delta$-quasipolar rings.

\begin{ex}\label{ilk}{\rm (1) Every semisimple ring and every Boolean ring is
$\delta$-quasipolar.\\
(2)  Since $J(\Bbb Q) = \delta(\Bbb Q) = \Bbb Q$, $\Bbb Q$ is
$\delta$-quasipolar. On the other hand,   $\Bbb Z$ is not
$\delta$-quasipolar since $\delta(\Bbb Z) = 0$. }
\end{ex}

One may suspects that every quasipolar ring is
$\delta$-quasipolar. But the following example erases the
possibility.

\begin{ex}\label{aksi} {\rm Let $p$ be a prime integer with $p\geq 3$ and $R = \Bbb
Z_{(p)}$ the localization of $\Bbb Z$ at the ideal $(p)$. By
\cite[Example 2.8]{CC3}, $R$ is a quasipolar ring. Since $J(R) =
\delta(R)$, it is not $\delta$-quasipolar. }\end{ex}

Let $S_r$ denote the right socle of the ring $R$, that is, $S_r$
is the sum of minimal right ideals of $R$.  We now prove that the
class of $J$-quasipolar rings is a subclass of $\delta$-quasipolar
rings.

\begin{lem}\label{iki}  If $R$ is a $J$-quasipolar ring, then  $R$ is $\delta$-quasipolar.
The converse holds if $S_r \subseteq J(R)$.
\end{lem}
\begin{proof}  The first assertion is clear since $J(R)\subseteq \delta(R)$.
Assume that $R$ is $\delta$-quasipolar. If $S_r \subseteq J(R)$,
then $J(R)/S_r=J(R/S_r)=\delta(R)/S_r$ \cite[Corollary 1.7]{Zh}
and we have $J(R)=\delta(R)$. Hence, $R$ is $J$-quasipolar.
\end{proof}

The converse of Lemma \ref{iki} is not true in general as the
following example shows.

\begin{ex}\label{ilki} Let $F$ be a
field and consider the ring
$R=\begin{bmatrix}F&F\\F&F\end{bmatrix}$. Then $R$ is a semisimple
ring and $R = \delta(R)$ and $J(R) = 0$. Hence $R$ is
$\delta$-quasipolar and it is not $J$-quasipolar.
\end{ex}

\begin{lem}\label{inv} Let $R$ be a ring. Then we have the following.
\begin{enumerate}
\item[(1)]  If $a, u\in R$ and $u$ is invertible, then $a$ is
$\delta$-quasipolar if and only if $u^{-1}au$ is
$\delta$-quasipolar.
\item[(2)] The element $a\in R$ is $\delta$-quasipolar if and only $-1-a$ is
$\delta$-quasipolar.
\item[(3)] If $R$ is a $\delta$-quasipolar ring with $\delta(R) =
J(R)$, then the spectral idempotent for any invertible element in
$R$ is the identity of $R$.
\end{enumerate}
\end{lem}

\begin{proof} (1) Assume that $a$ is $\delta$-quasipolar. Let $p^2
= p\in comm^2(a)$ such that $a + p\in \delta(R)$. Let $x\in
comm(u^{-1}au)$. Then $(uxu^{-1})a=a(uxu^{-1})$. Since $p\in
comm^2(a)$, $(uxu^{-1})p=p(uxu^{-1})$. Hence $(u^{-1}pu)^2 =
u^{-1}pu \in comm^2(u^{-1}au)$. Since $\delta(R)$ is an ideal of
$R$, $u^{-1}(a+p)u = u^{-1}au + u^{-1}pu\in\delta(R)$. Thus
$u^{-1}au$ is $\delta$-quasipolar. Conversely, if $u^{-1}au$ is
$\delta$-quasipolar, then by the preceding proof
$u(u^{-1}au)u^{-1} =
a$ is $\delta$-quasipolar.\\
(2) Assume that $a$ is $\delta$-quasipolar. Let $p^2 = p\in
comm^2(a)$ such that $a + p = r\in \delta(R)$. Then $-1-a +(1-p)
=-r\in \delta(R)$. Then $1-p\in comm^2(-1-a)$ and $1-p$ is the
spectral idempotent of $-1-a$. Conversely, if $-1-a$ is
$\delta$-quasipolar, then from what we have proved that $-1-(-1-a)
= a$ is quasipolar.\\
(3) Let $u$ be an invertible element in $R$. There exists $p^2 =
p\in comm^2(u)$ such that $u + p\in \delta(R)$. Then $1 +
u^{-1}p\in J(R)$. Hence $u^{-1}p$ and so $p$ is invertible. Thus
$p = 1$.
\end{proof}

In \cite[Corollary 2.3]{CC3}, it is proved that if  $R$ is a
$J$-quasipolar ring, then $2\in J(R)$. In this direction we prove
the following.

\begin{lem}\label{ucin}  If $R$ is a $\delta$-quasipolar ring, then $2\in
\delta(R)$. \end{lem}

\begin{proof} For the identity $1$, there exists $p^2 = p\in R$
such that $1 + p\in \delta(R)$. Multiplying the latter by $p$, we
have $2p\in \delta(R)$. Hence $1- p\in \delta(R)$. Thus $2\in
\delta(R)$. \end{proof}

Lemma \ref{ucin} can be used
to determine whether given rings are
$\delta$-quasipolar.

\begin{ex}\label{ikin}{\rm (1) The ring $\Bbb Z_3$ is a semisimple ring
and $\delta$-quasipolar but the ring $R = \begin{bmatrix}\Bbb
Z_3&\Bbb Z_3\\0&\Bbb Z_3\end{bmatrix}$ is not $\delta$-quasipolar
since $\delta(R) = \begin{bmatrix}0&\Bbb Z_3\\0&\Bbb
Z_3\end{bmatrix}$ and $2$ does not contained in
$\delta(R)$.\\
(2) Let $R = \{(a_{ij})\in $T$_n(\Bbb Z_3) \mid a_{11} = a_{22}=
\cdots = a_{nn}\}$. $\Bbb Z_3$ is $\delta$-quasipolar but $R$ is
not since $\delta(R) = \{(a_{ij})\in $T$_n(\Bbb Z_3) \mid a_{11} =
a_{22}= \cdots = a_{nn}= 0\}$ and $2$ does not contained in
$\delta(R)$. }
\end{ex}

Our next endeavor is to find conditions under which  a
$\delta$-quasipolar ring is quasipolar or vice versa. In this
direction recall that a ring $R$ is called {\it local} if it has
only one maximal left ideal, equivalently, maximal right ideal.

\begin{prop}\label{locquasi} Let $R$ be a local ring. If $R$ is quasipolar ring, then $R$ is $\delta$-quasipolar.
\end{prop}

\begin{proof}
Let $a\in R$. If $a\in J(R)$, it is clear. Assume that $a\notin
J(R)$. Since $R$ is local, $a$ is invertible. Hence $a+1 \in
\delta(R)$ by $\delta(R)=J(R)$.
\end{proof}

A ring $R$ is called {\it right principally projective} if every
principal right ideal is a direct summand. Principally projective
rings were introduced by Hattori \cite{Hat} to study the torsion
theory and play an important role in ring theory.

\begin{thm}\label{peh} Every $\delta$-quasipolar ring is right principally projective.
\end{thm}
\begin{proof} Assume that $R$ is a $\delta$-quasipolar ring. Let $a\in R$. By
Lemma \ref{ucin}, $2\in \delta(R)$ and $-1-a\in \delta(R)$. Then
$2 + (-1-a) = 1- a\in \delta(R)$. By \cite[Theorem 1.6 (5)]{Zh},
there exists a semisimple right ideal $Y$ of $R$ such that $R =
(1-(1-a))R\oplus Y = aR\oplus Y$.
\end{proof}

Recall that a ring $R$ is said to be {\it von Neumann regular} if
for any $a\in R$, there exists $b \in R$ with $a = aba$, while $R$
is  {\it strongly regular} if for any $a\in R$, there exists $b
\in R$ with $a = a^2b$. In this direction we prove the following.

\begin{lem}\label{strpireg} Every abelian $\delta$-quasipolar ring is strongly regular.
\end{lem}
\begin{proof} Let $a\in R$. By Theorem \ref{peh}, there exists a semisimple projective right ideal such that
$R = aR\oplus Y$. Then $1 = e + f$ where $e^2 = e\in aR$ and $f^2
= f\in Y$. So $eR\subseteq aR$ and $a = ea + fa$. Hence $a = ea
\in eR$. Thus $aR = eR$. It follows that $R$ is von Neumann
regular. Since $R$ is abelian, $R$ is strongly regular.\end{proof}

The following result is an immediate consequence of Lemma
\ref{strpireg}.

\begin{cor}\label{abdelta} Every abelian $\delta$-quasipolar ring is quasipolar.
\end{cor}

 A ring $R$ is said to be {\it clean} \cite{NC} if for each
$a\in R$ there exists $e^2 = e\in R$ such that $a -e$ is
invertible, and $R$ is called  {\it strongly clean} \cite{Ni}
provided that every element of $R$ can be written as the sum of an
idempotent and an invertible element that commute.

\begin{cor}\label{strclean} Every abelian $\delta$-quasipolar ring is strongly clean.
\end{cor}

The converse statement of Corollary \ref{strclean} need not hold
in general.

\begin{ex}{\rm Let $R = \{(q_1, q_2,q_3,\ldots , q_n, a, a, a, \ldots)\mid n\geq 1;
q_i\in\Bbb Q;$ $a\in \Bbb Z_{(2)}\}$.  Then $R$ is strongly clean
but not quasipolar (see \cite[Example 3.4(3)]{YC}). Therefore $R$
is not $J$-quasipolar since every $J$-quasipolar ring is
quasipolar.  On the other hand, since $S_r = 0$ and $\delta(R)/S_r
= J(R)/S_r$, $\delta(R) = J(R)$. Thus   $R$ is not
$\delta$-quasipolar.}
\end{ex}

In \cite[Theorem 2.9]{CC3}, it is shown that if the ring $R$  is
$J$-quasipolar, then $R/J(R)$ is Boolean and idempotents in
$R/J(R)$ lift $R$. We have the following result for
$\delta$-quasipolar rings.

\begin{thm}\label{besli1} If $R$ is a $\delta$-quasipolar ring,
then $R/\delta(R)$ is a Boolean ring and idempotents in
$R/\delta(R)$ lift $R$.
\end{thm}

\begin{proof} Let $\overline{a}\in
R/\delta(R)$. There exists $p^2 = p\in comm^2(-1+a)$ such that
$-1+a + p\in \delta(R)$. Hence $\overline{a} = \overline{1-p}$  is
an idempotent in $R/\delta(R)$ and $R/\delta(R)$ is a Boolean
ring. Let $\bar{a}^2 = \bar{a}\in R/\delta(R)$. Then $a^2- a\in
\delta(R)$. There exists $p^2 = p\in comm^2(a)$ such that $a +
p\in \delta(R)$. Multiplying $a + p\in \delta(R)$ from the right
by $p$ we have
\begin{equation} ap + p\in \delta(R) \end{equation}
Multiplying $a + p\in \delta(R)$ from the left by $a$ we have
\begin{equation}a^2 + ap\in
\delta(R) \end{equation}
 Substract (1) from
(2), we have $a^2- p\in \delta(R)$. Hence $a-p\in \delta(R)$ and
 so idempotents in $R/\delta(R)$ lift to $R$. \end{proof}

The concept of $\delta_r$-clean rings are defined in \cite{GO}. A
ring $R$ is called {\it $\delta_r$-clean} if for every element
$a\in R$ there  exists an idempotent $e\in R$ such that $a - e\in
\delta(R)$. A ring is {\it abelian} if all idempotents are
central.

\begin{lem}\label{sait} If $R$ is a $\delta$-quasipolar ring, then it is
$\delta_r$-clean. The converse holds if $R$ is abelian.
\end{lem}

\begin{proof}
 Let $R$ be a $\delta$-quasipolar ring and $a\in R$. There
exists $p^2 = p\in comm^2(-1+a)$ such that $-1+a + p\in
\delta(R)$. Then $a- (1- p)\in \delta(R)$. For the converse,
assume that $R$ is abelian. Let $a\in R$. There exists an
idempotent $e$ such that $1 + a- e\in \delta(R)$. By assumption,
$1-e$ is a central idempotent and so $1-e\in comm^2(a)$.
\end{proof}

 Recall that a ring $R$ is {\it exchange} if for every $a\in R$, there exists an idempotent $e\in aR$ such that
$1- e \in (1- a)R$. Namely, von Neumann regular rings and clean
rings are exchange.
\begin{cor}  Let $R$ be a $\delta$-quasipolar ring. Then
\begin{enumerate}\item[(1)] $R$ is an exchange ring.\item[(2)]
$R/\delta(R)$ is a clean ring.
\end{enumerate}
\end{cor}
\begin{proof}(1) Let $R$ be a $\delta$-quasipolar ring. By Lemma
\ref{sait}, $R$ is a $\delta_r$-clean ring. By \cite[Theorem
2.2(2)]{GO}, every $\delta_r$-clean ring is an exchange ring.\\(2)
By Theorem \ref{besli1}, $R/\delta(R)$ is Boolean, therefore, it
is clean.
\end{proof}
\begin{cor} Consider following conditions for a ring $R$.\begin{enumerate}
\item $R$ is $\delta$-quasipolar and $\delta(R) = 0$.\item $R$ is Boolean.\item $R$ is von Neumann regular and
$\delta$-quasipolar.\end{enumerate} Then \rm {(1)} $\Rightarrow$
\rm{(2)} $\Rightarrow$ \rm{(3)}.
\end{cor}
\begin{proof} (1) $\Rightarrow$ (2) Assume that $R$ is $\delta$-quasipolar and $\delta(R) =
0$. By Theorem \ref{besli1}, $R$ is Boolean.\\ (2) $\Rightarrow$
(3) Assume that $R$ is Boolean. Then it is commutative with
characteristic $2$ and $a^2 + a = 0\in \delta(R)$ and $a^2 = a =
a^3$ for all $a\in R$. Hence $R$ is von Neumann regular and
$\delta$-quasipolar.
\end{proof}

Let $R$ be a ring. The element $a\in R$ is called {\it $J$-clean}
if $a$ is the sum of an idempotent and a radical element in its
Jacobson radical. The ring $R$ is called {\it $J$-clean} if every
element is a sum of an idempotent and a radical element.

\begin{thm}\label{mete}
If $R$ is an abelian $J$-clean ring, then it is
$\delta$-quasipolar.
\end{thm}

\begin{proof}
Let $a\in R$. Then we have $-a \in R$. Since $R$ is $J$-clean,
there exist $e^2=e \in R$ and $j\in J(R)$ such that $-a = e + j$.
Hence $a+e \in J(R)$. Since $R$ is abelian, $e^2=e \in comm^2(a)$
and $J(R)\subseteq \delta(R)$, $R$ is  $\delta$-quasipolar as
asserted.
\end{proof}

All $\delta$-quasipolar rings  need not be Boolean and the
converse statement of Theorem \ref{mete} is not true in general.

\begin{ex}{\rm The ring $\mathbb{Z}_3$ is
semisimple and so $\mathbb{Z}_3 = \delta(\mathbb{Z}_3)$. Therefore
$\mathbb{Z}_3$ is $\delta$-quasipolar, but it is neither Boolean
nor $J$-clean.}
\end{ex}

In \cite[Proposition 2.11]{CC3}, it is shown that a ring $R$ is
local and $J$-quasipolar if and only if $R$ is $J$-quasipolar with
only trivial idempotents if and only if $R/J(R)\cong \Bbb Z_2$.
 We have the following  for
$\delta$-quasipolar rings.

\begin{prop}\label{temel}
Let $R$ be a ring with only trivial idempotents. Then $R$ is
$\delta$-quasipolar if and only if $R\cong \Bbb Z_2$ or
$R/\delta(R)\cong \Bbb Z_2$.
\end{prop}

\begin{proof}
Assume that $R$ is $\delta$-quasipolar. Let $a\in R$. There exists
an idempotent $p\in comm^2(a)$ such that $-a+p\in \delta(R)$. By
hypothesis $p = 1$ or $p = 0$. If $\delta(R) = 0$, then $R\cong
\Bbb Z_2$. Suppose that $\delta(R)\neq 0$. For any $a\in
R\setminus\delta(R)$, $\bar{a} = \bar{1}\in R/\delta(R)$. Hence
$R/\delta(R)\cong \Bbb Z_2$. Conversely, if $R\cong \Bbb Z_2$,
there is nothing to do. Suppose that $R/\delta(R)$ is isomorphic
to $\Bbb Z_2$ by isomorphism $f$. Let $a\in R\setminus \delta(R)$.
Then $f(-\overline{a}) = \overline{1}\in\Bbb Z_2$. Then
$f(-\overline{a}) = f(\overline{1})$ implies $-\overline{a}-
\overline{1}\in$ Ker$f = 0$. Hence $-\overline{a} = \overline{1}$.
That is, $a + 1\in\delta(R)$. Thus $R$ is $\delta$-quasipolar.
\end{proof}

Recall that a ring $R$ is called {\it strongly $\pi$-regular} if
for every element $a$ of $R$ there exist a positive integer $n$
(depending on $a$) and an element $x$ of $R$ such that $a^n =
a^{n+1}x$, equivalently, an element $y$ of $R$ such that $a^n =
ya^{n+1}$. In spite of the fact that $J(R)$ is contained in both
$\delta(R)$ and $R^{qnil}$, no comparings between $\delta(R)$ and
$R^{qnil}$ exist.  Strongly $\pi$-regular rings play crucial role
in this direction.

\begin{prop}\label{kiyas} Let $R$ be a $\delta$-quasipolar ring and $\delta(R)=J(R)$.
Then $R$ is strongly $\pi$-regular if and only if $J(R)=R^{qnil}=
nil(R)=\delta(R)$.
\end{prop}

\begin{proof} Necessity. Let $a\in R^{qnil}$. Then for any $x\in comm(a)$, $1-
ax$ is invertible. By hypothesis, there exist a positive integer
$m$ and $ b\in R$ such that $a^m = a^{m+1}b$. Since $b\in comm(a)$
by \cite[Page 347, Exercise 23.6(1)]{La}, $a^m = 0$. So $a\in
nil(R)$ or $R^{qnil} \subseteq nil(R)$. To prove $nil(R)\subseteq
\delta(R)$, let $a\in nil(R)$. By hypothesis there exists $p^2 =
p\in comm^2(1-a)$ such that $1-a+p\in \delta(R)$. Since $1- a$ is
invertible, $p = 1$ by Lemma \ref{inv} (3). Hence $2- a\in
\delta(R)$. Also $2\in \delta(R)$ by Lemma \ref{ucin}, we then
have $a\in \delta(R)$.

Sufficiency.  Let $a\in R$. There exists $p^2 = p\in comm^2(-1 +
a)$ such that $-1 + a + p\in \delta(R)$. Set $u =-1 + a + p\in
nil(R)$. Then $a + p$ is invertible and $ap = up$ is nilpotent so
that $a^np = 0$ for some positive integer $n$. So $a^n = a^n(1- p)
= (u + (1- p))^n(1- p) = (u + 1)^n(1- p) = (a + p)^n(1- p) = (1-
p)(a + p)^n$. By Nicholson, \cite[Proposition 1]{Ni}, $a$ is
strongly $\pi$-regular. This completes the proof.
\end{proof}

Let $R$ and $V$ be rings and $V$ be an $(R, R)$-bimodule that is
also a ring with $(vw)r = v(wr)$, $(vr)w = v(rw)$, and $(rv)w =
r(vw)$ for all $v$, $w\in V$ and $r\in R$. The \textit{Dorroh
extension} $D(R, V)$ of $R$ by $V$ defined as the ring consisting
of the additive abelian group $R\oplus V$ with multiplication $(r,
v)(s, w) = (rs, rw + vs + vw)$ where $r$, $s\in R$ and $v$, $w\in
V$.

Uniquely clean rings were introduced by Chen in \cite{Chen}. A
ring $R$ is {\it uniquely clean} in case for any $a \in R$ there
exists a unique idempotent $e \in R$ such that $a - e \in R$ is
invertible. In \cite{GO}, among others, uniquely $\delta_r$-clean
rings are studied. A ring $R$ is called {\it uniquely
$\delta_r$-clean} if for every element $a\in R$ there exists a
unique idempotent $e\in R$ such that $a- e\in \delta(R)$. Uniquely
clean Dorroh extensions in \cite[Proposition 7]{NZ} and uniquely
$\delta_r$-clean Dorroh extensions in \cite[Proposition 3.11]{GO}
are considered. Now we consider $\delta$-quasipolar Dorroh
extensions.

\begin{prop}\label{dor} Let $R$ be a ring. Then we have the
following.
\begin{itemize}\item[{\rm (1)}] If $D(R, V)$ is $\delta$-quasipolar,
then $R$ is $\delta$-quasipolar. \item[{\rm (2)}] If the following
conditions are satisfied, then $D(R, V)$ is $\delta$-quasipolar.
\begin{itemize}\item[{\rm(i)}] $R$ is $\delta$-quasipolar;
\item[{\rm(ii)}] $e^2 = e\in R$, then $ev = ve$ for all $v\in V$;
\item[{\rm(iii)}] $V = \delta(V)$.\end{itemize}
\end{itemize}
\end{prop}

\begin{proof} (1) Let $r\in R$. There exists $e^2 = e\in D(R, V)$ such that $e\in comm^2(r, 0)$ and
$(r,0) + e\in \delta(D(R, V))$. Since $e\in D(R, V)$, $e$ has the
form such that $(p, v)^2 = (p, v)$ and $p^2 = p$. Then $e = (p,
v)\in comm^2(r, 0)$ implies that $p\in comm^2(r)$ and $r + p\in
\delta(R)$ since $(r+p,v)\in \delta(D(R, V))$ and by
\cite[Proposition 3.11]{GO}. Hence $R$ is $\delta$-quasipolar.

(2) Assume that (i), (ii) and (iii) hold. Let $(r, v)\in D(R, V)$.
There exists $p^2 = p\in comm^2(r)$ such that  $r + p\in
\delta(R)$. By (iii), $(0, V)\subseteq \delta(D(R, V))$. Then $(r,
v) + (p, 0) = (r + p, v)\in \delta (D(R, V))$.
\end{proof}

As an application of Dorroh extensions we consider following
example. This example also shows that in the Proposition \ref{dor}
(2), the conditions (i), (ii) and (iii) are not superfluous.

\begin{ex}{\rm Consider the ring $D(\Bbb Z, \Bbb Q)$. Then $D(\Bbb Z, \Bbb Q)\cong \Bbb Z \times \Bbb
Q$. Then $\delta(\Bbb Z\times \Bbb Q) = (0)\times \Bbb Q$. Since
$\Bbb Z$ is not $\delta$-quasipolar,  $D(\Bbb Z, \Bbb Q)$ is not
$\delta$-quasipolar.}
\end{ex}

Let $R$ and $S$ be any ring and $M$ an $(R,S)$-bimodule. Consider
the ring of the formal upper triangular matrix ring $T =
\begin{bmatrix} R&M\\0&S\end{bmatrix}$. It is well known that
$\delta(T)\subseteq \begin{bmatrix}
\delta(R)&M\\0&\delta(S)\end{bmatrix}$. However, if $R = T = F$ is
a field, then $\delta(T) =
\begin{bmatrix} 0&F\\0&F\end{bmatrix}$.

The following example illustrate the  $\delta$-quasipolarity of
full matrix rings and upper triangular matrix rings depend on the
coefficient ring.

\begin{ex}\label{ustuc}\begin{enumerate}{\rm \item Consider the ring $R = \begin{bmatrix}\Bbb Z_2&\Bbb Z_2\\0&\Bbb
Z_2\end{bmatrix}$. Then $J(R) = \begin{bmatrix}0&\Bbb
Z_2\\0&0\end{bmatrix}$ and $\delta(R) =
\begin{bmatrix}0&\Bbb Z_2\\0&\Bbb Z_2\end{bmatrix}$. $R$ is
$\delta$-quasipolar.\item As noted in Example \ref{ikin}, the ring
$\Bbb Z_3$ is semisimple and therefore $\delta$-quasipolar.
However, the ring $\begin{bmatrix}\Bbb Z_3& \Bbb Z_3\\0&\Bbb
Z_3\end{bmatrix}$ is not $\delta$-quasipolar.
\item Let $A$ denote the matrix $\begin{bmatrix}$$0&1\\$$ 0&0$
$\end{bmatrix}\in{\rm Mat}_{2}(\Bbb Z)$. For any $P^2 = P\in
comm^2(A)$, the matrix $P$ has the form $P = $
$\begin{bmatrix}$$x&y\\$$ 0&x$ $\end{bmatrix}$ with $x^2 = x$ and
$2xy = y$ where $x, y\in \Bbb Z$. This would imply that $P$ is the
zero matrix or the identity matrix. Since $\delta(\Bbb Z) = 0$,
$\delta({\rm Mat}_{2}(\Bbb Z)) = 0$.  In consequence, $A + P$ can
not be in $\delta({\rm Mat}_{2}(\Bbb Z))$.  Therefore
Mat$_{2}(\Bbb Z)$ is not $\delta$-quasipolar.
\item Let $A =
\begin{bmatrix}$$1&0\\$$ 0&0$ $\end{bmatrix}\in T_{2}(\Bbb
Z)$. The idempotents of $T_2(\Bbb Z)$ are zero, identity,
$\begin{bmatrix}1&0\\0&0\end{bmatrix}$
$\begin{bmatrix}0&0\\0&1\end{bmatrix}$
$\begin{bmatrix}0&y\\0&1\end{bmatrix}$,
$\begin{bmatrix}1&y\\0&0\end{bmatrix}$ where $y$ is an arbitrary
integer. Since $A$ commutes with only zero, identity,
$\begin{bmatrix}1&0\\0&0\end{bmatrix}$ and
$\begin{bmatrix}0&0\\0&1\end{bmatrix}$, among these idempotents
there is no idempotent $P$ such that $A + P\in \delta(T_2(\Bbb
Z))$ since $\delta(T_2(\Bbb Z)) =
\begin{bmatrix}0&\Bbb Z\\0&0\end{bmatrix}$.
Hence $T_2(\Bbb Z)$ is not $\delta$-quasipolar. }
\end{enumerate}
\end{ex}

\section{Weakly $\delta$-quasipolar Rings} In this section, we introduce
an upper   class of $\delta$-quasipolar rings, namely,  weakly
$\delta$-quasipolar rings, and we give some properties of such
rings.

\begin{df}{\rm Let $R$ be a ring and $a\in R$. The element $a$ is called {\it
weakly $\delta$-quasipolar} if there exists $p^2 = p\in comm(a)$
such that $a + p\in \delta(R)$, and $p$ is called a {\it weakly
$\delta$-spectral idempotent}. Then $R$ is called {\it weakly
$\delta$-quasipolar} if every element of $R$ is weakly
$\delta$-quasipolar.}
\end{df}

An element of a ring is called {\it strongly $J$-clean} \cite{Ch}
provided that it can be written as the sum of an idempotent and an
element in its Jacobson radical that commute. A ring is {\it
strongly $J$-clean} in case each of its elements is strongly
$J$-clean.

\begin{ex}{\rm (1) Every semisimple ring and every Boolean ring is
weakly $\delta$-quasipolar, since $\delta$-quasipolar rings are
weakly
$\delta$-quasipolar.\\
(2) Every strongly $J$-clean ring is weakly $\delta$-quasipolar. }
\end{ex}

\begin{note}{\rm Let $R = \prod^n_{i=1}R_i$ be a finite direct
product of rings.
 It is easily checked that for any $a = (x_1, x_2,\dots ,x_n)\in R$,
 $comm(a) = \prod^n_{i=1}comm(x_i)$ and $a\in R$ is an idempotent
 in $R$ if and only if $x_i$ is an idempotent in $R_i$ for
 each $(i = 1, 2,\dots ,n)$.}
 \end{note}

\begin{prop}\label{don} Let $f : R\rightarrow S$ be a surjective ring
homomorphism. If $R$ is weakly $\delta$-quasipolar, then  $S$ is
weakly $\delta$-quasipolar.
\end{prop}
\begin{proof}  Let $s\in S$ with $s = f(r)$ where $r\in R$. There
exists an idempotent $p\in comm(r)$ such that $r + p\in
\delta(R)$. Let $q =f(p)$. Then $q^2 = q\in comm f(r) = comm(s)$.
By \cite{Zh}, $f(\delta(R))\subseteq \delta(S)$. Then $s + q =
f(r) + f(p) = f(r + p)\in f(\delta(R))\subseteq \delta(S)$. Hence
$S$ is weakly $\delta$-quasipolar.
\end{proof}

\begin{cor}\label{subring1} Every direct summand of a weakly $\delta$-quasipolar ring
is weakly $\delta$-quasipolar.
\end{cor}

The converse statement of Proposition \ref{don} need not hold in
general. There are rings $R$, $S$ and an epimorphism $f :
R\rightarrow S$ with $S$ is weakly $\delta$-quasipolar but $R$ is
not.
\begin{ex}\label{don1}{\rm Consider the field of rational numbers $\Bbb Q$ as a $\Bbb Z$-module.
There exists an index  set $I$ and an epimorphism $f :
\bigoplus_{i\in I}R_i\rightarrow \Bbb Q$ where $R_i = \Bbb Z$.
$\Bbb Q$ is weakly $\delta$-quasipolar but $\bigoplus_{i\in I}R_i$
is not weakly $\delta$-quasipolar since $\Bbb Z$ is not weakly
$\delta$-quasipolar and by Corollary \ref{subring1}.}
\end{ex}

\begin{prop}\label{dik} Let $R = \prod^n_{i=1}R_i$ be a finite direct product of rings.
$R$ is weakly $\delta$-quasipolar if and only if each $R_i$ is
weakly $\delta$-quasipolar for $(i=1,2,\dots ,n)$.
\end{prop}
\begin{proof} One way is clear  from Corollary \ref{subring1}. We
may assume that $n = 2$ and $R_1$ and $R_2$ are weakly
$\delta$-quasipolar. Let $a = (x_1, x_2)\in R$. There exist
idempotents $p_i\in comm(x_i)$ such that $x_i + p_i\in
\delta(R_i)$ for $(i=1,2)$. Then $p = (p_1, p_2)$ is an idempotent
in $R$ and $p\in comm(a)$ and $a + p\in \delta(R)$. Hence $R$ is
weakly  $\delta$-quasipolar.
\end{proof}

In \cite{GO}, Gurgun and Ozcan introduce and investigate
properties of $\delta_r$-clean rings. Motivated by this work
 strongly $\delta_r$-clean rings can be defined  as follows.

\begin{df} \rm An element $x\in R$ is called {\em strongly
$\delta_r$-clean} provided that there exist an idempotent $e\in R$
and an element $w\in \delta_r$ such that $x=e+w$ and $ew = we$. A
ring $R$ is called {\em strongly $\delta_r$-clean} in case every
element in $R$ is strongly $\delta_r$-clean.
\end{df}

Any strongly $J$-clean ring is strongly $\delta_r$-clean. But the
converse need not be true, for example any commutative semisimple
ring which is not a Boolean ring is such a ring.

Note that in the following theorem it is proved that  the notions
of strongly $\delta_r$-clean rings and weakly $\delta$-quasipolar
rings coincide.

\begin{thm}\label{burc} Let $R$ be a ring.
Then $R$ is a weakly $\delta$-quasipolar ring if and only if it is
strongly $\delta_r$-clean.\end{thm}

\begin{proof}  Let $R$
be a weakly $\delta$-quasipolar ring and $a\in R$. There exits
$p^2 = p\in comm(-1+a)$ such that $-1+a + p\in \delta(R)$. Then
$a-(1- p)\in \delta(R)$ and $a(1-p) = (1-p)a$. Hence $R$ is a
strongly $\delta_r$-clean ring. Conversely, assume that $R$ is is
a strongly $\delta_r$-clean ring. Let $a\in R$. Since $-a\in R$,
by assumption there exists an idempotent $p\in R$ such that $-a-
p\in\delta(R)$ and $(-a)p = p(-a)$. So $R$ is is a weakly
$\delta$-quasipolar ring.
\end{proof}

Theorem \ref{burc} states that the weakly $\delta$-quasipolarity
of a ring is equivalent to the strongly $\delta_r$-cleanness of
this ring. But the following example shows that there are rings in
which there exists a weakly $\delta$-quasipolar element which is
not strongly $\delta_r$-clean and there exists a strongly
$\delta_r$-clean element which is not weakly $\delta$-quasipolar.

\begin{ex}\rm{ Let $R=\Bbb Z$ and $a=1 \in R$. There exists
no idempotent $p$ such that $a+p \in \delta(R)$. Then $a$ is not
weakly $\delta$-quasipolar. Let $p=1 \in R$.  Since $a-p \in
\delta(R)$, $a$ is strongly $\delta_r$-clean. On the other hand,
if $a=-1 \in R$, then there exists no idempotent $p$ such that
$a-p \in \delta(R)$. Then $a$ is not strongly $\delta_r$-clean.
Let $p=1 \in R$. Since $a+p \in \delta(R)$, $a$ is weakly
$\delta$-quasipolar. }
\end{ex}

\begin{thm}\label{ogur} Let $R$ be a local ring with non-zero maximal ideal. Then the following are
equivalent.
\begin{enumerate}
\item[(1)] $R$ is weakly $\delta$-quasipolar;
\item[(2)] $R$ is strongly $J$-clean;
\item[(3)] $R$ is uniquely clean;
\item[(4)] $R/J(R)\cong \Bbb {Z}_2$;
\item[(5)] $R/\delta(R)\cong \Bbb {Z}_2$.
\end{enumerate}
\end{thm}

\begin{proof} Let $R$ be a local ring with non-zero
maximal ideal. \\
$(1)\Leftrightarrow(2)$   Assume that $R$ is weakly
$\delta$-quasipolar. Let $a\in R$. There exists $p^2=p \in
comm(-1+a)$ such that $-1+a+p \in \delta(R)$. Then $a-(1-p) \in
\delta(R)$. Since $p \in comm(-1+a)$, $pa=ap$. Hence $R$ is
strongly $J$-clean by $J(R)=\delta(R)$.
Similarly, the rest is clear.  \\
$(2)\Leftrightarrow(3)$ follows from \cite[Lemma 4.2]{Ch}.\\
$(3)\Leftrightarrow(4)$ follows from \cite[Theorem 15]{NZ}.\\
$(1)\Rightarrow(5)$ Let $R$ be weakly $\delta$-quasipolar and
$\overline{0}\neq\overline{a}=a+\delta(R)\in R/\delta(R)$, we show
that $\overline{a}=\overline{1}$. Then there exist an idempotent
$p\in R$  such that $-a+p \in \delta(R)$ and $p^2=p \in comm(-a)$.
Since $R$ is a local, $p=0$ or $p=1$. If $p=0$, this contradicts
$\overline{0}\neq\overline{a}$.   Therefore, $p=1$. It follows
that $\overline{a}=\overline{1}$.\\
$(5)\Rightarrow(1)$ It follows from Proposition \ref{temel}.
\end{proof}

\end{document}